\documentclass[a4paper,12pt]{amsart}
\usepackage{amsfonts,amsmath,amssymb,amsthm}
\usepackage{latexsym}
\usepackage{eucal}
\usepackage[dvips]{graphics}

\usepackage{mathabx}
\usepackage[english]{babel}
\usepackage{epsfig}
\usepackage[usenames]{color}
\usepackage[colorlinks=true, linkcolor=red, citecolor=blue]{hyperref}

\newcommand{\R}{{\mathbb R}}

\newcommand{\tr}{\mbox{tr}}

\newtheorem{theorem}{Theorem}[section]

\newtheorem{lemma}[theorem]{Lemma}

\newtheorem{proposition}[theorem]{Proposition}

\theoremstyle{definition}

\theoremstyle{remark}
\newtheorem*{remarks}{Remarks}

\newtheorem{assumptions}{Assumptions}

\makeatletter
\@namedef{subjclassname@2020}{\textup{2020} Mathematics Subject Classification}
\makeatother

\newcommand{\field}[1]{\mathbb{#1}}
\newcommand{\C}{\field{C}}

\newcommand{\M}{\mathcal{M}}
\newcommand{\remove}[1]{ }

\def\R{\mathbb R}
\def\be{\begin{equation}}
\def\ee{\end{equation}}
\def\ba{\begin{eqnarray}}
\def\ea{\end{eqnarray}}

\numberwithin{equation}{section}

\begin{document}
\title[Kawahara equation with boundary memory]
{\bf  Asymptotic behavior of Kawahara equation with memory effect}

\author[Capistrano-Filho]{Roberto de A.  Capistrano--Filho*}
\thanks{*Corresponding author: roberto.capistranofilho@ufpe.br}
\address{
Departamento de Matem\'atica, Universidade Federal de Pernambuco\\
Cidade Universit\'aria, 50740-545, Recife (PE), Brazil}
\email{roberto.capistranofilho@ufpe.br}
\thanks{Capistrano--Filho was supported by CNPq grants numbers  307808/2021-1, 401003/2022-1, and 200386/2022-0, CAPES  grants numbers 88881.311964/2018-01 and 88881.520205/2020-01, and MATHAMSUD 21-MATH-03.}


\author[Chentouf]{Boumedi\`ene Chentouf}
\address{
Faculty of Science, Kuwait University \\
Department of Mathematics, Safat 13060, Kuwait}
\email{boumediene.chentouf@ku.edu.kw}


\author[de Jesus]{Isadora Maria de Jesus}
\address{Departamento de Matem\'atica,\ Universidade Federal de Pernambuco (UFPE), 50740-545, Recife (PE), Brazil and Instituto de Matemática,\ cUniversidade Federal de Alagoas (UFAL),\ Maceió (AL), Brazil.}
\email{isadora.jesus@im.ufal.br; isadora.jesus@ufpe.br}


\begin{abstract}
In this work, we are interested in a detailed qualitative analysis of the Kawahara equation, a model that has numerous physical motivations such as magneto-acoustic waves in a cold plasma and gravity waves on the surface of a heavy liquid. First, we design a feedback control law, which combines a damping component and another one of finite memory-type. Then, we are capable of proving that the problem is well-posed under a condition involving the feedback gains of the boundary control and the memory kernel. Afterwards, it is shown that the energy associated with this system exponentially decays.
\end{abstract}

\date{\today}
\maketitle

\thispagestyle{empty}

\section{Introduction}
\subsection{Background and literature review}
Water wave models have been studies by many scientists from numerous disciplines such as hydraulic engineering, fluid mechanics engineering, physics as well as mathematics. These models are in general hard to derive, and complex to obtain qualitative information on the dynamics of the waves. This makes their studies interesting and challenging. Recently,  appropriate assumptions on the amplitude, wavelength, wave steepness, and so on, are invoked to investigate the asymptotic models for water waves and understand the full water wave system (see, for instance, \cite{ASL,BLS,Lannes} and references therein for a rigorous justification of various asymptotic models for surface and internal waves).

As a matter of fact, it has been noticed that the water waves can be considered as a free boundary problem of the incompressible, irrotational Euler equation in an appropriate non-dimensional form. This means that there are two non-dimensional parameters $\delta := \frac{h}{\lambda}$ and $\varepsilon := \frac{a}{h}$, where the water depth, the wavelength, and the amplitude of the free surface are respectively denoted by $h, \lambda$ and $a$. On the other hand, the parameter $\mu$, known as the Bond number, measures the importance of gravitational forces compared to surface tension forces. We also note that the long waves (also called shallow water waves) are mathematically characterised by the condition $\delta \ll 1$. Obviously, there are several long-wave approximations depending on the relation between $\varepsilon$ and $\delta$.

The above discussion led to, instead of studying models that do not give good asymptotic properties, we can rescale the parameters mentioned above to find systems that reveal asymptotic models for surface and internal waves, like the Kawahara model.  Precisely, letting $\varepsilon = \delta^4 \ll 1$, $\mu = \frac13 + \nu\varepsilon^{\frac12}$, and the critical Bond number $\mu = \frac13$,  the so-called equation Kawahara equation is put forward. Such an equation was derived by Hasimoto and Kawahara \cite{Hasimoto1970,Kawahara}  and takes the form
\[\pm2 W_t + 3WW_x - \nu W_{xxx} +\frac{1}{45}W_{xxxxx} = 0,\]
or, after re-scaling,
\begin{equation}\label{fda1}
	W_{t}+\alpha W_{x}+\beta W_{xxx}-W_{xxxxx}+WW_{x}=0.
\end{equation}
The latter is also seen as the fifth-order KdV equation \cite{Boyd}, or singularly perturbed KdV equation \cite{Pomeau}. It describes a dispersive partial differential equation with numerous wave physical phenomena such as magneto-acoustic waves in a cold plasma \cite{Kakutani}, the propagation of long waves in a shallow liquid beneath an ice sheet \cite{Iguchi}, gravity waves on the surface of a heavy liquid \cite{Cui}, etc.

Note that valuable efforts in the last decays were made to understand this model in various research frameworks. For example, numerous works focused on the analytical and numerical methods for solving \eqref{fda1}. These methods include the tanh-function method \cite{Berloff}, extended tanh-function method \cite{Biswas}, sine-cosine method \cite{Yusufoglu}, Jacobi elliptic functions method \cite{Hunter}, direct algebraic method and numerical simulations \cite{Polat}, decompositions methods \cite{Kaya}, as well as the variational iterations and homotopy perturbations methods \cite{Jin}. Another direction is the study of the Kawahara equation from the point of view of control theory and specifically, the controllability and stabilization problem \cite{ara}, which is our motivation.

Whereupon, we are interested in a detailed qualitative analysis for the system \eqref{fda1} in a bounded region. More precisely, our primary concern is to analyze the qualitative properties of solutions to the initial-boundary value problem associated with the model \eqref{fda1} posed on a bounded interval under the presence of damping and memory-type controls.

Now, we will present some previous results that dealt with the asymptotic behavior of solutions for the Kawahara model \eqref{fda1} in a bounded domain. One of the first outcomes is due to Silva and Vasconcellos \cite{vasi1,vasi2}, where the authors studied the stabilization of global solutions of the linear Kawahara equation, under the effect of a localized damping mechanism. The second endeavor is completed by Capistrano-Filho \textit{et al.} \cite{ara}, where the generalized Kawahara equation in a bounded domain is considered, that is, a more general nonlinearity $W^p \partial_x W$ with $p\in [1,4)$ is taken into account. It is proved that the solutions of the Kawahara system decay exponentially when an internal damping mechanism is applied.

Recently,  a new tool for the control properties of the Kawahara operator was proposed in \cite{CaSo}. In this work, the authors treated the so-called overdetermination control problem for the Kawahara equation. Precisely, a boundary control is designed so that the solution to the problem under consideration satisfies an integral condition. Furthermore, when the control acts internally in the system, instead of the boundary, the authors proved that this integral condition is also satisfied.

We conclude the literature review with three recent works. In \cite{CaVi, boumediene} the stabilization of the Kawahara equation with a localized time-delayed interior control is considered. Under suitable assumptions on the time delay coefficients, the authors proved that the solutions of the Kawahara system are exponentially stable. This result is established by means of either the Lyapunov method or a compactness-uniqueness argument. More recently, the Kawahara equation in a bounded interval and with a delay term in one of the boundary conditions was studied in \cite{luan}. The authors used two different approaches to prove that the solutions of \eqref{fda1} are exponentially stable under a condition on the length of the spatial domain. We point out that this is a small sample of papers that were concerned with the stabilization problem of the Kawahara equation in a bounded interval. Of course, we suggest that the reader, who is interested in more details on the topic, consult the papers cited above and the references therein.

Let us now present the framework of this article.
\subsection{Problem setting and main results}
Consider the system \eqref{fda1} in a bounded domain $\Omega=(0,\ell)$, where $\ell > 0$ is the spatial length, under the action of the following feedback:
\begin{equation}\label{sis1}
\begin{cases}
	\begin{aligned}
		\partial_{t} \omega(t,x)+&\alpha \partial_x \omega(t,x) +\beta\partial_x^3 \omega(t,x)- \partial_x^5 \omega(t,x)\\
		& + {\omega^p}(t,x) \partial_x \omega(t,x)=0,
	\end{aligned} & x\in \Omega,\ t>0, \\
	\omega (t,0) =\omega (t,\ell) =0 ,& t>0, \\
   \partial_x \omega(t,0)=\partial_x \omega(t,\ell) =0,& t>0, \\
	\partial_x^2 \omega(t,\ell)=\mathcal{F}(t),& t>0,\\
	\partial_x^{2}\omega(t,0)=z_0(t), &  t\in \mathcal{I},\\
	\omega(0,x) =\omega_{0} (x), & x \in\Omega,
\end{cases}
\end{equation}
with $\omega_0$, $z_0$ are initial data and the feedback law is a linear combination of the damping and finite memory terms given by 
\begin{equation}\label{fdl}
\mathcal{F}(t):=\nu_1 \partial_x^2 \omega(t,0)+\nu_2 \int_{t-\tau_2}^{t-\tau_1} \sigma(t-s)  \partial_x^2 \omega (s,0) \, ds.
\end{equation}
Here, $\alpha >0$ and $\beta>0$ are physical parameters, $p \in \{1,2\}$, whereas $\nu_1$ and $\nu_2$ are nonzero real numbers. In turn, $0<\tau_1 < \tau_2$ correspond to the finite memory interval $(t-\tau_1, t-\tau_2)$. Moreover, $\mathcal{I}=( -\tau_2, 0 )$, and the memory kernel is denoted by $\sigma(s)$. Of course, the presence of a memory term is usually ubiquitous in practice. Particularly, memory is of great importance in systems control as they are governed by equations, where the past values of one or more variables involved in the system play a crucial role. On the other hand, the impact of a memory term in some systems can be deleterious as it can affect their performance \cite{bc1, bc2, nipi}. Last but not least, we indicate that the memory term, that arises in the boundary control \eqref{fdl}, could reflect the case of a compressible (or incompressible) viscoelastic fluid. The latter is regarded as the simplest material with memory \cite{afg, da}.

On the other hand, let us note that the energy associated with the full system \eqref{sis1} is given by
\begin{equation}\label{energia} 
    \mathcal{E}(t)= \int_{\Omega}\omega^2 (t,x)dx + |\nu_2| \int_{\mathcal{M}} s \sigma(s) \left( \int_{\Omega_0}   (\partial_x^2 \omega)^2 (t-s \phi,0) \, d\phi \right)ds, \ t \geq 0.
\end{equation}
Naturally, as we are interested in the behavior of the system \eqref{sis1}, we need to check whether the feedback law, given by \eqref{fdl}, represents a damping mechanism. In other words, we would like to see if, in the presence of the boundary memory-type feedback law, the energy of the system \eqref{energia} tends to zero state with some specific decay rate, when $t$ goes to $0$.  This situation can be presented in the following question:

\vspace{0.2cm}
\noindent\textbf{Question:} \textit{Does $\mathcal{E}(t)\longrightarrow0$, as $t\to\infty$? If it is the case, is it possible to come up with a decay rate?}
\vspace{0.2cm}

To answer the previous question for the system \eqref{sis1}, we will assume, from now on, that the memory kernel $\sigma$ obeys the following conditions:
\begin{assumptions}\label{assK}
The function $\sigma \in \ell^{\infty} (\mathcal{M})$, where $\mathcal{M}:=(\tau_1, \tau_2)$. In turn, we assume that
$$\sigma (s) > 0,\quad \text{a.e. in}\  \mathcal{M}.$$
Moreover, the feedback gains $\nu_1$ and $\nu_2$ together with the memory kernel satisfy
\begin{equation}\label{ab}
|\nu_1| + |\nu_2| \displaystyle \int_{\mathcal{M}} \sigma(s) \, ds<1.
\end{equation}
\end{assumptions}
\vspace{2mm}

Some notations, that we will use throughout this manuscript,  are presented below:
\begin{itemize}
\item[(i)] We denote by $(\cdot,\cdot)_{\mathbb{R}^{2}}$ the canonical inner product of $\mathbb{R}^{2}$, whereas $\langle\cdot,\cdot\rangle$ denotes the canonical inner product of $\ell^2(\Omega)$ whose induced norm is $\|\cdot\|$.

	\item[(ii)] For $T>0$, consider the space of solutions
	\begin{equation*}
		Y_{T}= C(0,T;L^2(\Omega))\cap L^{2}(0,T; { H_0^2}(\Omega))
	\end{equation*}
equipped with the norm
\begin{equation*}
\|v\|_{Y_{T}}^2= \left(\max_{t \in (0,T)}\|v(t,\cdot )\| \right)^2 + \int_{0}^{T}\|v(t,\cdot )\|^{2}_{{ H_0^2}(\Omega)}dt.
\end{equation*}

	\item[(iii)] Let $\Omega_0=(0,1)$ and $\mathcal{Q}:=\Omega_0 \times \mathcal{M}$. Then, we consider the spaces
$$H:=L^2 (\Omega) \times L^2 (\mathcal{Q}), \quad \mathcal{H}:=L^2 (\Omega) \times L^2 (\mathcal{I} \times \mathcal{M}),$$
respectively equipped with the following inner product:
$$
\left\{ \begin{array}{l}
\displaystyle \langle(\omega,z),(v,y)\rangle_{H}=\langle \omega,v\rangle +|\nu_2| \int_{\mathcal{M}} \int_{\Omega_0}  s \sigma(s) z(\phi,s)y(\phi,s)  ~ d\phi ds,  \\[4mm]
\displaystyle  \langle(\omega,z),(v,y)\rangle_\mathcal{H}=\langle \omega,v\rangle + |\nu_2|  \int_{\mathcal{I}} \int_{\mathcal{M}} \sigma(s) z(r,s)y(r,s)\, ds dr.
\end{array}
\right.
$$

\end{itemize}

Subsequently, we can state our first main result:
\begin{theorem}\label{Lyapunov}Under the assumptions \ref{assK} and assuming that the length $\ell$ fulfills the smallness condition
\begin{equation}\label{L}
0< \ell < \pi \sqrt{\frac{3\beta}{\alpha}},
\end{equation}
there exists $r>0$ sufficiently small, such that for every $(\omega_{0}, z_{0})\in H$ with $\|(\omega_{0}, z_{0})\|_{H} < r$, the energy of the system \eqref{sis1}, given by \eqref{energia}, is exponentially stable. In other words, there exist two positive constants $\kappa$ and $\mu$ such that
\begin{equation}\label{exp decay}
	\mathcal{E}(t) \leq \kappa E(0)e^{-2\mu t}, \ t > 0,
\end{equation}
where $\mathcal{E}(t)$ is defined by \eqref{energia}.
\end{theorem}
The proof of this result uses an appropriate Lyapunov function, which requires the condition \eqref{L}. In turn, such a requirement can be relaxed by using a compactness-uniqueness argument \cite{ro} (see \cite{ara,luan,cajesus,vasi1,vasi2}).  The proof is based on the following outcome \cite{luan}:
\begin{lemma}\label{lem2}
Let $\ell>0$ and consider the assertion: There exist \; $\zeta \in \C$  and $\omega  \in H_0^2(\Omega)\cap H^5(\Omega)$ such that
\begin{equation*}
\begin{cases}
\begin{array}{ll}
\zeta \omega(x) +\omega'(x)+\omega'''(x)-\omega'''''(x)=0, & x \in  \Omega,  \\
\omega(x)=\omega'(x)=\omega''(x)=0, & x \in \{0,\ell\}.
\end{array}
\end{cases}
\label{len}
\end{equation*}
If $(\zeta,\omega) \in \C \times  H_0^2(\Omega)\cap H^5(\Omega)$ is solution of \eqref{len}, then
$\omega=0.$
\end{lemma}

We have:

\begin{theorem}\label{comp}
Suppose that assumptions \ref{assK} hold. Moreover, we choose $\ell>0$ so that the problem in Lemma \ref{lem2} has only the trivial solution. Then, there exists $\varrho>0$ such that for every $\left(\omega_{0}, z_{0}\right) \in H$ satisfying
$
\left\|\left(\omega_{0}, z_{0}\right)\right\|_{H} \leq \varrho,
$
the energy \eqref{energia} of the problem \eqref{sis1} decays exponentially. 
\end{theorem}

\subsection{Further comments and paper's outline} As mentioned above, the exponential stability result of the system \eqref{sis1} will be established using two different methods. The first one evokes a Lyapunov function and requires an explicit smallness condition on the length of the spatial domain $\ell$. The second one is obtained via a classical compactness-uniqueness argument, where critical lengths phenomena appear with a relation with the M\"{o}bius transforms (see for instance \cite{luan}). This permits us to answer the question raised in the introduction.

\begin{remarks}
Let us point out some important comments:
\begin{itemize}

\item[$\bullet$] Considering $\nu_2=0$ and $\alpha=0$, the authors in \cite{cajesus} showed the stabilization property for \eqref{sis1} using the compactness-uniqueness argument. Since they removed the drift term $\alpha\partial_{x}\omega$, the critical lengths phenomena did not appear.

\vspace{0.2cm}

\item[$\bullet$] The main concern of this work is to deal with the feedback law of memory type as in \eqref{fdl}. In fact, one needs to control this term to ensure well-posedness and stabilization results.

\vspace{0.2cm}

\item[$\bullet$] Our results are valid for the general nonlinearities $u^{p}\partial_{x}u$, $p\in\{1,2\}$, and also can be extended for linearity like $c_1u \partial_x u+c_2u^2 \partial_x u$. To draw more attention to the first general nonlinearity, the decay rate in \eqref{exp decay} depends on the values of $p$ since we have (see Section 3)

 $$\mu<\min\left\{\dfrac{\mu_2 |\nu_2|e^{-\delta\tau_2}\delta}{2(1+\mu_1|\nu_2|)},
 \dfrac{\mu_1}{2\ell^2(1+\ell\mu_1)(p+2)}\left[(p+2)(3\pi^2\beta-\alpha \ell^2)-2\pi^2\ell^{2-\frac{p}{2}}r^p\right]\right\}.$$
\end{itemize}
\end{remarks}

\bigskip
We end our introduction with the paper's outline:  The work consists of three parts including the Introduction. Section \ref{Sec2} discusses the existence of solutions for the full system \eqref{sis1}. Section \ref{Sec3} is devoted to proving the stabilization results, that is, Theorem \ref{Lyapunov} and  Theorem \ref{comp}. 

\section{Well-posedness theory}\label{Sec2}
In this section, we are interested in analyzing the well-posedness property of the system \eqref{sis1}. The first and the second subsections are devoted to proving the existence of solutions for the linearized (homogenous and non-homogeneous) system associated with \eqref{sis1}, respectively. The third subsection concerns the well-posedness of the full system \eqref{sis1}.
\subsection{Linear problem}
As in the literature (see for instance the references  \cite{xyl} and \cite{np}),  the homogenous linear system associated with \eqref{sis1} can be viewed as follows:
 \begin{equation}\label{sis2.1_coupled}
	\left\{
	\begin{array}{ll}
		\partial_{t} \omega(t,x)+\alpha \partial_x \omega(t,x) +\beta\partial_x^3 \omega(t,x)- \partial_x^5 \omega(t,x) =0, & (t,x) \in  \mathbb{R}^{+}\times\Omega,\\
		s\partial _tz(t,\phi,s)+\partial_\phi z(t,\phi,s)=0, & (t,\phi,s)\in \R^+\times\Omega_0\times \mathcal{M},\\
		\omega (t,0) =\omega (t,\ell) =\partial_x \omega(t,0)=\partial_x \omega(t,\ell) =0, & t >0,\\
\partial_x^2 \omega(t,\ell)=\nu_1 \partial_x^2 \omega(t,0)+	\nu_2 \displaystyle \int_{\mathcal{M}} \sigma(s)  z(t,1,s) \, ds, & t >0,\\
\omega(0,x) =\omega_{0} (x), & x \in  \Omega,\\
z(0,\phi,r)=z_0(-\phi r),& (\phi,r)\in \Omega_0\times(0,\tau_2),
	\end{array}
	\right.
\end{equation}
where $z(t,\phi,s)=  \partial_x^2\omega(t - \phi s,0)$ satisfies a transport equation (see \eqref{sis2.1_coupled}$_2$). Letting $\Lambda(t)= \left[\begin{array}{ll}
		\omega(t,\cdot) \\
	    z(t,\cdot,\cdot)
	\end{array}\right],	\Lambda_{0}= \left[\begin{array}{ll}
	\omega_{0} \\
	z_{0}(-\phi\cdot)
\end{array}\right]$, one can rewrite this system abstractly:
\begin{equation}
\label{sla}
\begin{cases}
\Lambda_t(t)=A\Lambda(t), \quad t>0,\\
\Lambda(0)=\Lambda_0\in H,
\end{cases}
\end{equation}
where
\begin{equation*}
	 A= \left[\begin{array}{lll}
		-\alpha\partial_x - \beta\partial_x^3 + \partial_x^5 & 0 \\
		0 & -\dfrac{1}{s}\partial_{\phi}
	\end{array}\right],
\end{equation*}
whose domain is given by
\begin{equation*}\label{DAA}
D(A) :=
\left\lbrace
	\begin{aligned}
		&(\omega,z)\in H, \\
		&(\omega,z)\in H^5(\Omega)\cap H_0^2(\Omega), \\
		&z  \in L^2 \Bigl(  \M; H^1(\Omega_0)  \Bigr);
	\end{aligned}
	\left\lvert
	\begin{aligned}
	&\partial_x^{2}\omega(0)= z(0,\cdot),\\
	& \partial_x^{2}\omega(\ell)= \nu_1\partial_x^{2}\omega(0) +\nu_2  \int_{\mathcal{M}} \sigma(s) z(1,s) \, ds
	\end{aligned}
	\right.
\right\rbrace.
\end{equation*}

The following result ensures the well-posedness of the linear homogeneous system.

\begin{proposition}\label{linear}
Under the assumption \eqref{assK}, we have:
\begin{itemize}
\item[i.] The operator $A$ is densely defined in $H$ and generates a $C_{0}$-semigroup of contractions $e^{t{\scriptstyle A}}$. Thereby, for each $\Lambda_0\in H$, there exists a unique mild solution $\Lambda\in C([0,+\infty),H)$ for the linear system associated with \eqref{sis1}. Moreover, if $\Lambda_0\in D(A)$, then we have a unique classical solution with the regularity $$\Lambda\in C([0,+\infty),D(A))\cap C^1([0,+\infty),H).$$
\item[ii.] Given  $\Lambda_{0}=(\omega_{0},z_0(\cdot)) \in H$, the following estimates hold:
\begin{equation}
\label{est1}
\displaystyle \| \partial_x^2 \omega (0,\cdot) \|_{L^2 (0,T)}^2 + \int_0^T \int_{\mathcal{M}} s \sigma(s)   z^2 (t,1,s) \, dsdt \leq  C  \| (\omega_0, z_0(\cdot) )\|_H^2,
\end{equation}
\begin{equation}
\label{est2}
\displaystyle \| \partial_x^2 \omega (\cdot) \|_{L^2 (0,T;L^2 (\Omega) )}^2  \leq  C  \| (\omega_0, z_0(\cdot) )\|_H^2,
\end{equation}
\begin{equation}
\label{est3}
\displaystyle \| z_0(\cdot) \|_{L^2 (\mathcal{Q})}^2  \leq    \| z (T,\cdot,\cdot) \|_{L^2 (\mathcal{Q})}^2 +   \displaystyle \int_0^T \int_{\mathcal{M}}  \sigma(s)   z^2 (t,1,s) \, dsdt,
\end{equation}
and
\begin{equation}
\label{est4}
  T\| \omega_0(\cdot) \|^2  \leq    \| \omega\|_{L^2 (0,T;L^2(\Omega))}^2 +T\|\partial^2_x \omega(0)\|_{L^2 (0,T)}^2.
\end{equation}
\item[iii.]The map $$\mathcal{G}: \Lambda_{0}=(\omega_{0},z_0(\cdot)) \in H \mapsto \Lambda (\cdot) =e^{\cdot {\scriptstyle A}} \Lambda_{0} \in Y_{T} \times C \left( [0,T]; \, L^2 (\mathcal{Q}) \right)$$ is continuous.
\end{itemize}
\end{proposition}
\begin{proof}
\noindent\textbf{Proof of item i.}  This part can be proved by using the semigroup theory. In fact, note first that for given $\Lambda=(\omega,z)\in D(A),$ it follows from the Cauchy-Schwarz inequality that
\begin{equation}
\label{dcs}
\int_{\mathcal{M}}\sigma(s)z(1,s)ds\leq \left(\int_{\mathcal{M}}\sigma(s)ds\right)^\frac{1}{2}\left(\int_{\mathcal{M}}\sigma(s)(z(1,s))^2ds\right)^\frac{1}{2}.
\end{equation}
Thus, using integration by parts and \eqref{dcs} yields that
\begin{equation}
\label{dissipatividade_de_A}
\begin{split}
 \langle A\Lambda, \Lambda\rangle=&\dfrac{1}{2}\left[\left(\nu_1\partial_x^2 \omega(0)+\nu_2\int_{\mathcal{M}}\sigma(s)z(1,s)~ ds\right)^2-\left(\partial_x^2 \omega(0)\right)^2 \right.\\
& \left.-|\nu_2|\int_{\mathcal{M}}\sigma(s)\left(z(1,s)\right)^2~ ds+|\nu_2|\left(\partial_x^2 \omega(0)\right)^2\int_{\mathcal{M}}\sigma(s)~ ds\right]\\
\leq&\dfrac{1}{2}\left[\left(\partial_x^2 \omega(0)\right)^2\left(\nu_1^2-1+|\nu_2|\int_{\mathcal{M}}\sigma(s)~ ds\right)\right.\\
&\left.+2\nu_1\nu_2\left(\partial_x^2 \omega(0)\right)\left(\int_{\mathcal{M}}\sigma(s)z(1,s)~ ds\right)\right.\\
& \left.+\left(\nu_2^2-\dfrac{|\nu_2|}{\|\sqrt{\sigma(s)}\|^2}\right)\left(\int_{\mathcal{M}}\sigma(s)z(1,s)~ ds\right)^2\right]
=\dfrac{1}{2}\langle G X,X \rangle _{\R^2},
\end{split}
\end{equation}
where $$X=\left(\begin{array}{c}
\partial_x^2 \omega(0)\\
\displaystyle \int_{\mathcal{M}}\sigma(s)z(1,s)~ ds
\end{array}\right)$$
and
$$G=\left(\begin{array}{cc}
\displaystyle \nu_1^2-1+|\nu_2|\displaystyle \int_{\mathcal{M}}\sigma(s)~ ds &\nu_1\nu_2\\
\displaystyle \nu_1\nu_2&\displaystyle \nu_2^2-\dfrac{|\nu_2|}{\|\sqrt{\sigma(s)}\|^2}
\end{array}\right).$$
Due to \eqref{ab}, we have
$$\det G=|\nu_2|\left(\int_{\mathcal{M}}\sigma(s)~ ds\right)^{-1}\left\{\left[1-|\nu_2|\left(\int_{\mathcal{M}}\sigma(s)~ ds\right)\right]^2-\nu_1^2\right\}>0$$
and
$$
\tr G\leq |\nu_1|(|\nu_1|-1)-|\nu_1||\nu_2|\left(\int_{\mathcal{M}}\sigma(s)~ ds\right)^{-1}<0,
$$
since $|\nu_1|<1.$ Moreover, is not difficult to see that $G$ is a negative definite matrix. Putting these previous information together in \eqref{dissipatividade_de_A} we have that $A$ is dissipative. Analogously, considering the adjoint operator of $A$ as follows
$$A^*(v,y)=\left(\alpha\partial_x v+\beta \partial_x^3 v-\partial_x^5 v, \dfrac{1}{s}\partial_\phi y\right)$$ with domain
\begin{equation*}\label{DAAA}
D(A^*) :=
\left\lbrace
	\begin{aligned}
		&(v,y)\in H, \\
		& (\omega,z)\in H^5(\Omega)\cap H_0^2(\Omega), \\
		&y  \in L^2 \Bigl(  \M; H^1(\Omega_0)  \Bigr);
	\end{aligned}
	\left\lvert
	\begin{aligned}
	& \partial_x^2 v(\ell)=\dfrac{|\nu_2|}{\nu_2}y(1,s),\\
	& \partial_x^2 v(0)=\displaystyle \nu_1\partial_x^2 v(\ell) +|\nu_2|\int_{\mathcal{M}}\sigma(s)y(0,s)~ ds
	\end{aligned}
	\right.
\right\rbrace,
\end{equation*}
we have that for $(v,y)\in D(A^*),$
 \begin{equation}
 \label{dissipatividade_adjunto}
 \begin{split}
 \langle A^*(v,y),(v,y)\rangle
& +\left[|\nu_2|^2-|\nu_2|\|\sqrt{\sigma}\|^2_{L^2(\M)}\right]\left(\int_{\mathcal{M}}\sigma(s) y(0,s) ds\right)^2\\
=& \dfrac{1}{2}\langle G_*Z,Z\rangle,
 \end{split}
 \end{equation}
 where $$Z=\left(\begin{array}{c}  \partial_x^2 v(\ell)\\
 \displaystyle \int_{\mathcal{M}}\sigma(s) y(0,s) ds
 \end{array}\right)$$ and
 $$G_*=\left(\begin{array}{cc}
 \nu_1^2-1+|\nu_2|\displaystyle \int_{\mathcal{M}}\sigma(s)~ ds &\nu_1|\nu_2|\\
 \nu_1|\nu_2|&\displaystyle \nu_2^2-\dfrac{|\nu_2|}{\|\sqrt{\sigma(s)}\|^2}
\end{array}\right).$$
Again, thanks to the relation \eqref{ab}, we have $\det G_*=\det G>0$ and $ \tr G_*= \tr G<0$, since $|\nu_1|<1.$ Thus, using the fact that $G_*$ is negative definite in \eqref{dissipatividade_adjunto}, we have that $A^*$ is also dissipative, showing the item i.

\vspace{0.2cm}
\noindent\textbf{Proof of item ii.}  First, remember that $e^{tA}$ is a contractive  semigroup and therefore, for each $\Lambda_0=(\omega_0,z_0)\in H,$ the following estimate is valid
\begin{equation}
\label{contration_semigroup}
\|(\omega(t),z(t,\cdot,\cdot))\|_H^2=\|\omega(t)\|^2+\|z(t,\cdot,\cdot) \|^2_{L^2(\mathcal{Q})}\leq \|\omega_0\|^2+\|z_0(-\cdot) \|^2_{L^2(\mathcal{Q})}, \forall t\in [0,T].
\end{equation}
Moreover, the following inequality holds
\begin{equation}
\label{eq2.16}
\begin{array}{rcl}
\displaystyle \int_0^T\int_{\mathcal{M}}s\sigma(s)\left[z(t,1,s)\right]^2~dsdt&\leq&\displaystyle \dfrac{\tau_2}{|\nu_2|}\int_{\Omega_0}\int_{\mathcal{M}}|\nu_2| s \sigma(s)\left[z_0^2(-\phi s)\right]dsd\phi\\
&&\displaystyle +\dfrac{\tau_2}{\tau_1|\nu_2|}\int_0^T\int_{\Omega_0}\int_{\mathcal{M}}|\nu_2| s \sigma(s)z^2~ds d\phi dt.
\end{array}
\end{equation}
Indeed, multiplying the second equation of \eqref{sis2.1_coupled} by $\phi \sigma(s)z,$ rearranging the terms, integrating by parts and taking into account that $s \in \M=(\tau_1,\tau_2),$ we have
\begin{equation*}
\begin{split}
 \int_0^T\int_{\mathcal{M}}s\sigma(s) \left( z(t,1, s)\right)^2~ds dt
 \leq&\dfrac{\tau_2}{|\nu_2|}\int_0^T\int_{\Omega_0}\int_{\mathcal{M}}|\nu_2|\sigma(s) \left( z(t,\phi, s)\right)^2~ds d\phi dt\\
& +\dfrac{\tau_2}{|\nu_2|}\int_{\Omega_0}\int_{\mathcal{M}}\phi |\nu_2|\sigma(s)s\left( z(0,\phi, s)\right)^2~ds d\phi\\
&-\dfrac{\tau_2}{|\nu_2|}\int_{\Omega_0}\int_{\mathcal{M}}|\nu_2|\phi \sigma(s)s\left( z(T,\phi, s)\right)^2~ds d\phi\\
\leq& \dfrac{\tau_2}{\tau_1|\nu_2|}\int_0^T\int_{\Omega_0}\int_{\mathcal{M}}s|\nu_2|\sigma(s) \left( z(t,\phi, s)\right)^2~ds d\phi dt\\
& +\dfrac{\tau_2}{|\nu_2|}\int_{\Omega_0}\int_{\mathcal{M}}\phi |\nu_2|\sigma(s)s\left( z_0(-\phi s)\right)^2~ds d\phi
\end{split}
\end{equation*}
This proves the estimate \eqref{eq2.16}. As a consequence of \eqref{contration_semigroup}, \eqref{eq2.16} and the hypothesis of $\tau_1\leq s\leq \tau_2$ and $\phi\leq 1$, we also have
\begin{equation}
\label{eq2.15}
\int_0^T\int_{\mathcal{M}} s\sigma(s)\left(z(t,1,s)\right)^2~dsdt\leq \dfrac{\tau_2}{|\nu_2|}\left(\dfrac{T}{\tau_1}+1\right)\left(\|\omega_0\|^2+\| z_0(-\phi s)\|^2_{L^2(\mathcal{Q})}\right).
\end{equation}

Now, we are in a position to prove \eqref{est1}. Multiplying the first equation of  \eqref{sis2.1_coupled} by $\omega$, integrating over $[0,T]\times[0,\ell],$ and using the boundary conditions, it follows that

\begin{equation}\label{eq2.23}
\begin{split}
\|\partial_x^2\omega(0)\|_{L^2(0,T)}^2=&\displaystyle \|\omega_0\|^2+\int_0^T\left(\partial_x^2\omega(\ell)\right)^2~dt-\|\omega(T)\|^2\\
\leq& \|\omega_0\|^2+\int_0^T\left(\nu_1\partial_x^2\omega(0)+\nu_2\int_{\mathcal{M}}
\sigma(s)z(\cdot,1, s)ds \right)^2~dt:=I_{1}+I_{2}.
\end{split}
\end{equation}
To estimate the integral $I_{2}$ on the right-hand side of \eqref{eq2.23}, we use Young's inequality together with the Cauchy-Schwartz inequality, to obtain
\begin{equation}
\label{eq2.25}
\begin{split}
I_{2}\leq&\nu_1^2\left(\partial_x^2 \omega(t,0)\right)^2\\
&+2|\nu_1||\nu_2|\left(\partial_x^2 \omega(t,0)\right)\left(\int_{\mathcal{M}}\sigma(s)ds\right)^\frac{1}{2}\left(\int_{\mathcal{M}}\sigma(s)z^2(\cdot,1, s)ds\right)^\frac{1}{2}\\
& +\nu_2^2\left(\left(\int_{\mathcal{M}}\sigma(s)ds\right)^\frac{1}{2}
\left(\int_{\mathcal{M}}\sigma(s)z^2(\cdot,1, s)ds\right)^\frac{1}{2}\right)^2\\
\leq&  \left[\nu_1^2+\dfrac{\nu_2^2}{2\theta}\left(\int_{\mathcal{M}}\sigma(s)ds\right)\right]\left(\partial_x^2 \omega(t,0)\right)^2\\
&+\left[2\theta \nu_1^2+\nu_2^2\left(\int_{\mathcal{M}}\sigma(s)ds\right)\right]\left(\int_{\mathcal{M}}\sigma(s)z^2(\cdot,1, s)ds\right).
\end{split}
\end{equation}
Thereafter, inserting \eqref{eq2.25} into \eqref{eq2.23}, we find
\begin{equation}
\label{eq2.26}
\begin{split}
&\left[1-\nu_1^2-\dfrac{\nu_2^2}{2\theta}\left(\int_{\mathcal{M}}\sigma(s)ds\right)\right]\|\partial_x^2\omega(0)\|_{L^2(0,T)}^2\leq \|\omega_0\|^2 \\&+\left[2\theta\nu_1^2+\nu_2^2\left(\int_{\mathcal{M}}\sigma(s)ds\right)\right]\left(\int_0^T\int_{\mathcal{M}}\sigma(s)z^2(\cdot,1, s)dsdt\right).
\end{split}
\end{equation}
Thanks to \eqref{ab}, one can choose $\theta>0$ large enough so that
\begin{equation}\label{w1}
\displaystyle 1-\nu_1^2-\dfrac{\nu_2^2}{2\theta}\left(\int_{\mathcal{M}}\sigma(s)ds\right)>0.
\end{equation}
This, together with \eqref{eq2.26} and \eqref{eq2.15}, yields
\begin{equation}
\label{eq2.28}
\begin{split}
 \|\partial_x^2\omega(0)\|_{L^2(0,T)}^2\leq \leq&  \displaystyle C\left(\|\omega_0\|^2+\dfrac{1}{\tau_1}\int_0^T\int_{\mathcal{M}}s\sigma(s)z^2(\cdot,1, s)dsdt\right)\\
\leq&   C\left(1+\dfrac{\tau_2}{\tau_1|\nu_2|}\left(\dfrac{T}{\tau_1}+1\right)\right)\|\omega_0\|^2+\dfrac{C\tau_2}{\tau_1|\nu_2|}\left(\dfrac{T}{\tau_1}+1\right)\|z_0(-\phi s)\|^2_{L^2(\mathcal{Q})}\\
\leq&  \displaystyle C\left(\|\omega_0\|^2+\|z_0(-\phi s)\|^2_{L^2(\mathcal{Q})}\right).
\end{split}
\end{equation}
Clearly, combining \eqref{eq2.15} and \eqref{eq2.28}, we get \eqref{est1}.

Now, let us prove \eqref{est2}. Multiplying the equation \eqref{sis2.1_coupled} by $xu$, integrating by parts over $(0,T)\times \Omega,$ and isolating the term $\|\partial_x^2 \omega\|_{L^2(0,T;L^2(\Omega))}^2$, we obtain
\begin{equation*}
\begin{split}
  \|\partial_x^2 \omega\|_{L^2(0,T;L^2(\Omega))}^2
\leq&  \int_{\Omega} \dfrac{x}{5}\omega_0^2(x)dx+\dfrac{\alpha}{5}\| \omega\|_{L^2(0,T;L^2(\Omega))}^2 \\
& +\dfrac{\ell}{5}\left[\nu_1^2+\dfrac{\nu_2^2}{2\epsilon}\left(\int_{\mathcal{M}}\sigma(s)ds\right)\right]\int_0^T(\partial_x^2\omega(t,0))^2\\
& +\dfrac{\ell}{5}\left[2\epsilon\nu_1^2+\nu_2^2\left(\int_{\mathcal{M}}\sigma(s)ds\right)\right]\int_0^T\int_{\mathcal{M}}\sigma(s)z^2(t,1,s)dsdt\\
\leq&  \dfrac{\ell}{5}\|\omega_0\|^2+\dfrac{\alpha}{5}\| \omega\|_{L^2(0,T;L^2(\Omega))}^2 \\
&+C_1\left[\int_0^T(\partial_x^2\omega(t,0))^2+\int_0^T\int_{\mathcal{M}}\sigma(s)z^2(t,1,s)dsdt\right],
\end{split}
\end{equation*}
where \eqref{eq2.25} is used and
$$ C_1=\max\left\{\dfrac{\ell}{5}\left[\nu_1^2+\dfrac{\nu_2^2}{2\epsilon}\left(\int_{\mathcal{M}}\sigma(s)ds\right)\right],\dfrac{\ell}{5}\left[2\epsilon\nu_1^2+\nu_2^2\left(\int_{\mathcal{M}}\sigma(s)ds\right)\right]\right\}.$$ Now, taking into account the fact that $e^{At}$ is a  semigroup of contractions and using  \eqref{est1}, we obtain \eqref{est2} with the constant $C=\max\left\{\dfrac{\ell}{5},\dfrac{\alpha}{5}, C_1\right\}$.

Finally, let us show \eqref{est3} and \eqref{est4}, respectively. For \eqref{est3}, multiply the second equation in \eqref{sis2.1_coupled} by $\sigma(s)z$ and integrates by parts over  $(0,T)\times\mathcal{Q},$ to obtain
$$
\int_{\Omega_0}\int_{\mathcal{M}}s\sigma(s) z^2(0,\phi, s) ~ds d\phi\leq \int_{\Omega_0}\int_{\mathcal{M}}s\sigma(s) z^2(T,\phi, s) ~ds d\phi +\int_0^T\int_{\mathcal{M}}\sigma(s) z^2(t,1,s) ~dsdt,
$$
showing \eqref{est3}. To prove \eqref{est4}, we multiply the first equation in \eqref{sis2.1_coupled} by $2(T-t)\omega$ and integrating over $[0,T]\times[0,\ell],$ to find
$$T\|\omega_0\|^2\leq T\|\omega\|^2_{L^2(0,T;L^2(\Omega))}+T\int_0^T\left(\partial_x^2\omega(0)\right)^2dt,$$
giving \eqref{est4}. Last but not least, it is worth mentioning that the above estimates remain true for solutions stemming from $\Lambda_0\in H,$ giving item ii.

\vspace{0.2cm}
\noindent\textbf{Proof of item iii.} Follows directly from \eqref{est2} and from \eqref{contration_semigroup}.
\end{proof}
\subsection{Non-homogeneous problem}
Let us now consider the linear system \eqref{sis2.1_coupled} with a source term $f \in L^1(0,T; L^2(\Omega))$ in the right-hand side of the first equation. As done in the previous subsection, the system can be rewritten as follows:
\begin{equation}\label{sis_coupled_nh}
\begin{cases}
\Lambda_t(t)=A\Lambda(t)+(\varphi(t,\cdot),0), \quad t>0,\\
\Lambda(0)=\Lambda_0\in H,
\end{cases}
\end{equation}
where $\Lambda=(\omega,z)$ and $\Lambda_ 0=(\omega_ 0,z_ 0(-\cdot)).$ With this in hand, the following result will be proved.
\begin{theorem}
\label{teor2}
Under the assumption \eqref{assK}, it follows that:
\begin{enumerate}
\item[$(a)$] If $\Lambda_0=(\omega_0,z_0(-\cdot))\in H$ and $\varphi \in L^1(0,T; L^2(\Omega)),$ then there exists a unique mild solution $$\Lambda=(\omega,z)\in Y_{T}\times C([0,T];L^2(\mathcal{Q}))$$ of \eqref{sis_coupled_nh} such that
\begin{equation}
\label{eq2.38} \|(\omega,z)\|_{C([0,T];H)}^2\leq C\left(\|(\omega_0,z_0(-\cdot))\|_H^2+\|\varphi\|_{L^1(0,T; L^2(\Omega))}^2\right),
\end{equation}
and
\begin{equation}
\label{M} \|\omega\|_{Y_{T}}^2\leq C\left(\|(\omega_0,z_0(-\cdot))\|_H^2+\|\varphi\|_{L^1(0,T; L^2(\Omega))}^2\right),
\end{equation}
for some constant $C>0$, which is independent of $\Lambda_0$ and $\varphi.$
\item[$(b)$] Given $$\omega\in Y_{T}=C(0,T;L^2(\Omega))\cap L^2(0,T;H^2_0(\Omega))$$ and $p\in\{1,2\}$, we have $\omega^p\partial_x \omega\in L^1(0,T; L^2(\Omega))$ and the map
\begin{equation}
\label{Theta}\mathcal{F}: \omega\in Y_{T} \mapsto \omega^p\partial_x \omega\in L^1(0,T; L^2(\Omega))
\end{equation}
is continuous.
\end{enumerate}
\end{theorem}
\begin{proof}
\vspace{0.2cm}
\noindent\textbf{Proof of item (a).}  Since $A$ is the infinitesimal generator of a semigroup of contractions $e^{tA}$ and $\varphi\in L^1(0,T;L^2(\Omega))$ it follows from semigroups theory that there is a unique mild solution $\Lambda=(\omega,z)\in C([0,T];H)$ of \eqref{sis_coupled_nh} such that
$$\Lambda(t)=e^{tA}\Lambda_0+\int_0^te^{(t-s)A}(\varphi,0)ds$$
and hence, we get
$$\|(\omega,z)\|_{C([0,T];H)}\leq C\left(\|(\omega_0,z_0(-\cdot))\|_H+\|\varphi\|_{L^1(0,T; L^2(\Omega))}\right).$$
Young's inequality gives
$$ \|(\omega,z)\|_{C([0,T];H)}^2\leq   2C^2\left(\|(\omega_0,z_0(-\cdot))\|_H^2+\|\varphi\|_{L^1(0,T; L^2(\Omega))}^2\right),$$
which proves \eqref{eq2.38}.  To complete the proof of item $(a)$, we must verify the validity of \eqref{M}. For this, observe that from \eqref{eq2.38}, we have
\begin{equation}
\label{eq2.39}
\max_{t\in[0,T]}\|\omega\|^2\leq  2C^2\left(\|(\omega_0,z_0(-\cdot))\|_H^2+\|\varphi\|_{L^1(0,T; L^2(\Omega))}^2\right).
\end{equation}
In turn, if we multiply the second equation in \eqref{sis_coupled_nh} by $\phi\sigma(s)z$, integrating over $[0,T]\times [0,1]\times [\tau_1,\tau_2 ]$ and arguing as for the proof of \eqref{eq2.16}, we obtain
\begin{equation}
\label{eq2.42}
\begin{array}{l}
\displaystyle\int_0^T\int_{\mathcal{M}} s\sigma(s)\left(z(t,1,s)\right)^2~dsdt\\
\leq\displaystyle \dfrac{\tau_2}{|\nu_2|}\left(\dfrac{T}{\tau_1}+1\right)\left(\|\omega_0\|^2+\| z_0(-\phi s)\|^2_{L^2(\mathcal{Q})}+\|\varphi\|_{L^1(0,T; L^2(\Omega))}^2\right).
\end{array}
\end{equation}
Now, multiplying the first equation in \eqref{sis_coupled_nh} by $\omega$, integrating over $[0,T]\times[0,\ell],$  and thanks to \eqref{eq2.42}, we get
\begin{equation}
\label{eq2.46}
\begin{split}
\|\partial_x^2\omega(0)\|_{L^2(0,T)}^2\leq& \|\omega_0\|^2+\int_0^T\left(\nu_1\partial_x^2\omega(0)+\nu_2\int_{\mathcal{M}}\sigma(s)z(\cdot,1, s)ds \right)^2~dt\\
& +2\left(\max_{t\in [0,T]}\|\omega(t,x)\|\right) \int_0^T\|\varphi(t,x)\|~dt.
\end{split}
\end{equation}
Now, replacing \eqref{eq2.25} in \eqref{eq2.46}, we find
\begin{equation}
\begin{array}{rcl}
\|\partial_x^2\omega(0)\|_{L^2(0,T)}^2&\leq& \displaystyle \|\omega_0\|^2+\left[\nu_1^2+\dfrac{\nu_2^2}{2\theta}
\left(\int_{\mathcal{M}}\sigma(s)ds\right)\right]\int_0^T\left(\partial_x^2 \omega(t,0)\right)^2 dt\\
&&\displaystyle +\left[2 \theta \nu_1^2+\nu_2^2\left(\int_{\mathcal{M}}\sigma(s)ds\right)\right]\left(\int_0^T\int_{\mathcal{M}}\sigma(s)z^2(\cdot,1, s)dsdt\right)\\
&&\displaystyle +2\left(\max_{t\in [0,T]}\|\omega(t,x)\|\right) \int_0^T\|\varphi(t,x)\|~dt.
\end{array}
\end{equation}
Isolating $\|\partial_x^2\omega(0)\|_{L^2(0,T)}^2$ and using Young's inequality for the last term of the right-hand side, we reach
\begin{equation}
\label{eq2.48}
\begin{array}{l}
\displaystyle \left[1-\nu_1^2-\dfrac{\nu_2^2}{2\theta}\left(\int_{\mathcal{M}}\sigma(s)ds\right)\right]\|\partial_x^2\omega(0)\|_{L^2(0,T)}^2\\
\leq \displaystyle \|\omega_0\|^2 +\left[2\theta \nu_1^2+\nu_2^2\left(\int_{\mathcal{M}}\sigma(s)ds\right)\right]
\left(\int_0^T\int_{\mathcal{M}}\sigma(s)z^2(\cdot,1, s)dsdt\right)\\[3mm]
\displaystyle\ \ \ \  +\left(\max_{t\in [0,T]}\|\omega(t,x)\|\right)^2+\|\varphi\|_{L^1(0,T;L^2(\Omega))}^2.
\end{array}
\end{equation}
Thanks to \eqref{ab}, \eqref{w1} and \eqref{eq2.48}, the estimate \eqref{eq2.38} becomes
\begin{equation}\label{eq2.50}
\begin{split}
 \|\partial_x^2\omega(0)\|_{L^2(0,T)}^2
 \leq& C_1\left(2+C_2+\dfrac{\tau_2}{\tau_1|\nu_2|}\left(\dfrac{T}{\tau_1}+1\right)\right)\|\omega_0\|^2\\
   & +C_1\left(\dfrac{\tau_2}{\tau_1|\nu_2|}\left(\dfrac{T}{\tau_1}+1\right)+1+C_2\right)\|z_0(-\phi s)\|^2_{L^2(\mathcal{Q})}\\
   & +C_1(1+C_2)\|\varphi\|_{L^1(0,T;L^2(\Omega))}^2\\
 \leq&  C\left(\|(\omega_0,z_0(-\phi s))\|^2_{H}+\|\varphi\|_{L^1(0,T;L^2(\Omega))}^2\right).
\end{split}
\end{equation}
Now, multiply the equation \eqref{sis_coupled_nh} by $xu$ and integrate by parts over $(0,T)\times (0,\ell ) $ and then perform similar calculations to those of the previous item to get
\begin{equation}
\label{eq2.53}
\begin{split}
&\dfrac{5}{2}\|\partial_x^2 \omega\|_{L^2(0,T;L^2(\Omega))}^2\leq \displaystyle \dfrac{\ell}{2}\|\omega_0\|^2+\dfrac{aT}{2}C\left(\|(\omega_0,z_0(-\phi s))\|_{H}^2+\|\varphi\|^2_{L^1(0,T;L^2(\Omega))}\right)\\
& +\dfrac{\ell}{2}C\left(\|(\omega_0,z_0(-\phi s))\|_{H}^2+\|\varphi\|^2_{L^1(0,T;L^2(\Omega))}\right)+\dfrac{\ell}{2}\|\varphi\|_{L^1(0,T;L^2(\Omega))}^2\\
& +\dfrac{\ell}{2}\left[\nu_1^2+\dfrac{\nu_2^2}{2\epsilon}\left(\int_{\mathcal{M}}\sigma(s)ds\right)\right]C\left(\|(\omega_0,z_0(-\phi s))\|_{H}^2+\|\varphi\|^2_{L^1(0,T;L^2(\Omega))}\right) \\
& +\dfrac{\ell}{2\tau_1}\left[2\epsilon\nu_1^2+\nu_2^2\left(\int_{\mathcal{M}}\sigma(s)ds\right)\right]\dfrac{\tau_2}{|\nu_2|}\left(\dfrac{T}{\tau_1}+1\right)\left(\|(\omega_0,z_0(-\phi s))\|_{H}^2+\|\varphi\|^2_{L^1(0,T;L^2(\Omega))}\right),
\end{split}
\end{equation}
where we have used Cauchy-Schwarz inequality, Young inequality, estimates \eqref{eq2.25}, \eqref{eq2.42}, and \eqref{eq2.50}.  Therefore, taking any $\epsilon>0$ in \eqref{eq2.53}, there exists $C>0$ such that
\begin{equation}
\label{eq2.54}
\begin{array}{l}
\displaystyle
\|\omega\|_{L^2(0,T;H^2_0(\Omega))}^2=\|\partial_x^2 \omega\|_{L^2(0,T;L^2(\Omega))}^2\leq C\left(\|(\omega_0,z_0(-\phi s))\|_{H}^2+\|\varphi\|^2_{L^1(0,T;L^2(\Omega))}\right).\\
\end{array}
\end{equation}
The estimate \eqref{M} follows directly from the estimates \eqref{eq2.39} and \eqref{eq2.54}, and item (a) is achieved.

\vspace{0.2cm}
\noindent\textbf{Proof of item (b).} Given $\omega,v\in Y_{T}$ we have, for $p=1,$ that
  \begin{equation}
  \label{eq2.55}
    \|\omega\partial_x \omega\|_{L^1(0,T;L^2(\Omega))}\leq k\int_0^T\|\omega\|_{L^2(\Omega)}\|\partial_x \omega\|dt\leq k\int_0^T\|\omega\|_{H^ 2(\Omega)}^2dt\leq k\|\omega\|_{Y_{T}}^2<\infty,
  \end{equation}
   where $k$ is the positive constant of the Sobolev embedding $L^2(\Omega)\hookrightarrow L^\infty(\Omega)$. Therefore, $\omega\partial_x \omega\in L^1(0,T;L^2(\Omega)),$ for each $\omega\in Y_{T}. $ Thus, using the triangle inequality, together with the Cauchy-Schwarz inequality, we get the classical estimate
  \begin{equation}
  \label{eq2.56}
 \|\mathcal{F}(\omega)-\mathcal{F}(v)\|_{L^1(0,T;L^2(\Omega))}\leq  k\|\omega-v\|_{Y_{T}}\left(\| \omega\|_{Y_{T}} +\|v\|_{Y_{T}}\right), \quad \text{for any} \, u,v \in Y_{T}.
 \end{equation}
 Therefore, the map $\mathcal{F}$ is continuous concerning the corresponding topologies. On the other hand, when $p=2,$ we have for $\omega,v\in Y_{T}$ that
\begin{equation}
\label{eq2.58}
\|\mathcal{F}(\omega)\|_{L^1(0,T;L^2(\Omega))}\leq k\|\omega\|_{C(0,T; L^2(\Omega))}\int_{0}^T \|\omega\|_{H^2(\Omega)}^2dt\leq k\|\omega\|_{Y_{T}}^3<+\infty.\end{equation}
 Hence, $\mathcal{F}(\omega)$ is well-defined and {for any} $u,v$ in $Y_{T}$, we have
\begin{equation}
\label{eq2.59}
\begin{split}
 \|\mathcal{F}(\omega)-\mathcal{F}(v)\|_{L^1(0,T;L^2(\Omega))}
 \leq & \dfrac{3k}{2}\left(\|\omega\|_{Y_{T}}^2+ \|v\|_{Y_{T}} ^2\right)\|\omega-v\|_{Y_{T}}.\\
\end{split}
\end{equation}
 Thereby, the map $\mathcal{F}$ is continuous for the corresponding topologies. 
 \end{proof}

\subsection{Nonlinear problem} We are now in a position to prove the main result of the section. Precisely, the next result gives the well-posedness for the full system  \eqref{sis1}.
\begin{theorem}
\label{theorem3} Suppose that \eqref{ab} holds. Then, there exist constants $r,C>0$ such that, for every $\Lambda_0=(\omega_0,z_0(-\cdot))\in H$ with $\|\Lambda_0\|^2_H\leq r,$ the problem \eqref{sis1} admits a unique global solution $\omega\in Y_{T},$ which satisfies $\|\omega\|_{Y_{T}}\leq C\|\Lambda_0\|_H.$
\end{theorem}
\begin{proof}
Given $\Lambda_0=(\omega_0,z_0(-\cdot))\in H$ such that $\|\Lambda_0\|^2_H\leq r,$ where $r$ is a positive constant to be chosen, define a mapping $\Upsilon:Y_{T}\rightarrow Y_{T}$ as follows: $\Upsilon(\omega)=y,$ where $y$ is the solution of \eqref{sis_coupled_nh} with a source term $\varphi=\omega^p\partial_x \omega=\mathcal{F}(\omega),\ p\in\{1,2\}$. The mapping $\Upsilon$ is well defined because of item $(a)$ of Theorem \ref{teor2} from which we obtain from \eqref{M} that
$$\|\Upsilon(\omega)\|_{Y_{T}}^2\leq C\left(\|\Lambda_0\|_H^2+\|\mathcal{F}(\omega)\|^2_{L^1(0,T:L^2 (\Omega))}\right).$$
Note that $\Upsilon(\omega)-\Upsilon(v)$ is a solution of \eqref{sis_coupled_nh} with initial condition $\Lambda_0=(0,0)\in H$ and source term $\varphi=\mathcal{F}(\omega)-\mathcal{F} (v).$ It follows from \eqref{M} that
$$\|\Upsilon(\omega)-\Upsilon(v)\|_{Y_{T}}^2\leq C\|\mathcal{F}(\omega)-\mathcal{F}(v)\|_{L^1(0,T:L^2 (\Omega))}^2,$$
where the constant $C>0$ above does not depend on $\Lambda_0$ and $\varphi.$

Now, considering $p=1$, we have from \eqref{eq2.55} that
$$\|\Upsilon(\omega)\|_{Y_{T}}^2\leq C\left(r+k^2\|\omega\|_{Y_{T}}^4\right),\ \forall \omega\in Y_{T},$$
while from \eqref{eq2.56}, we have that
$$\|\Upsilon(\omega)-\Upsilon(v)\|_{Y_{T}}^2\leq Ck^2\left(\|\omega\|_{Y_{T}}^2+\|v\|_{Y_{T}}^2\right)^2\|\omega-v\|_{Y_{T}}^2,\ \forall \omega, v \in Y_{T}.$$
Thus, when $\|\omega\|_{Y_{T}}^2\leq R$ we get
\begin{equation}
\label{eq3.31}
\begin{array}{rcl}
\|\Upsilon(\omega)\|_{Y_{T}}^2&\leq&C\left(r+k^2R^2\right),\ \forall \omega\in \mathcal{B},\\
&&\\
\|\Upsilon(\omega)-\Upsilon(v)\|_{Y_{T}}^2&\leq&4Ck^2R^2 \|\omega-v\|_{Y_{T}}^2,\ \forall \omega, v \in \mathcal{B}.
\end{array}
\end{equation}
Next, pick $R=\dfrac{1}{5k^2C}$ and $r=\dfrac{1}{25k^2C^2}$. For  $\omega\in\mathcal{B}=\{\omega\in Y_{T}; \|\omega\|_{Y_{T}}^2\leq R\},$ we have that
\begin{equation}
\label{eq3.32}
\begin{array}{rcl}
\|\Upsilon(\omega)\|_{Y_{T}}^2&\leq& R,\ \forall \omega\in\mathcal{B},\\
&&\\
\|\Upsilon(\omega)-\Upsilon(v)\|_{Y_{T}}^2&\leq&\dfrac{4}{5}\|\omega-v\|_{Y_{T}}^2,\ \forall \omega, v \in\mathcal{B}.
\end{array}
\end{equation}
On the other hand, when $p=2$, we have from \eqref{eq2.58} that
$$\|\Upsilon(\omega)\|_{Y_{T}}^2\leq C\left(r+k^2\|\omega\|_{Y_{T}}^6\right),\ \forall \omega\in Y_{T}$$
and from \eqref{eq2.59}, we have that
$$\|\Upsilon(\omega)-\Upsilon(v)\|_{Y_{T}}^2\leq C\left(\dfrac{3k}{2}\right)^2\left(\|\omega\|_{Y_{T}}^2+\|v\|_{Y_{T}}^2\right)^2\|\omega-v\|_{Y_{T}}^2,\ \forall \omega, v \in Y_{T}.$$
Thus, when $\|\omega\|_{Y_{T}}^2\leq R$, we get
\begin{equation}
\label{eq3.33}
\begin{array}{rcl}
\|\Upsilon(\omega)\|_{Y_{T}}^2&\leq&C\left(r+k^2R^3\right),\ \forall \omega\in \mathcal{B},\\
&&\\
\|\Upsilon(\omega)-\Upsilon(v)\|_{Y_{T}}^2&\leq&9Ck^2R^2 \|\omega-v\|_{Y_{T}}^2,\ \forall \omega, v \in\mathcal{B}.
\end{array}
\end{equation}
Therefore, just take $R=\dfrac{1}{4k\sqrt{C}}$ and $r=\dfrac{1}{16kC^\frac{3}{2}}$ and we will have that
\begin{equation}
\label{eq3.34}
\begin{array}{rcl}
\|\Upsilon(\omega)\|_{Y_{T}}^2&\leq& R,\ \forall \omega\in\mathcal{B},\\
&&\\
\|\Upsilon(\omega)-\Upsilon(v)\|_{Y_{T}}^2&\leq&\dfrac{9}{16}\|\omega-v\|_{Y_{T}}^2,\ \forall \forall \omega, v \in\mathcal{B}.
\end{array}
\end{equation}
Consequently, due to \eqref{eq3.32} and \eqref{eq3.33}, the restriction of the map $\Lambda$ to $\mathcal{B}$ is well-defined, and $\Lambda$ is a contraction on the ball $\mathcal{B}.$ As an application of Banach Fixed Point Theorem, the map $\Lambda$ possesses a unique fixed element $\omega,$ which turns out to be the unique solution to problem \eqref{sis1}. 
Finally, the solution is global thanks to the dissipation property. Indeed, the energy $\mathcal{E}(t)$ (see \eqref{energia}) of the system \eqref{sis1} satisfies
$$
\mathcal{E}^{\prime}(t) \leq \frac{1}{2}\langle G X, X\rangle_{\mathbb{R}^2} \leq 0,
$$
where $G$ and $X$ are given in Proposition \ref{linear}.
\end{proof}

\section{Exponential stability of solutions}\label{Sec3}
In this section, we will prove the two main results of our work. The first stabilization result will be proved \textit{via} the Lyapunov approach. The second one is obtained showing an \textit{observability inequality} which will be proved by the compactness-uniqueness argument.
\subsection{Proof of Theorem \ref{Lyapunov}}
Initially, let us remember that the energy of the system \eqref{sis_coupled_nh}, for $\varphi=\omega^p\partial_x \omega$, with $p\in\{1,2\}$, is defined by
$$\displaystyle \mathcal{E}(t)=\|\Lambda(t)\|_H^2=\|\omega(t)\|^2+\|z(t)\|^2_{L^2(\mathcal{Q })},$$
where $\displaystyle \|z(t)\|^2_{L^2(\mathcal{Q})}=|\nu_2|\int_{\mathcal{M}}s\sigma(s)\int_{0}^{1}z^2(t,\phi,s)d\phi ds.$ Thus, using \eqref{sis_coupled_nh}, we get
\begin{equation}
\label{derivative energi}
\begin{array}{rcl}
\mathcal{E}^{\prime}(t)&=&2\langle\Lambda_t(t),\Lambda(t)\rangle_H=2\langle A\Lambda(t),\Lambda(t)\rangle_H+2\langle(\omega^p\partial_x \omega,0),\Lambda(t)\rangle_H\\
&=&\displaystyle \langle GX,X\rangle_{\R^2}+2\int_{\Omega} u^{p+1}\partial_x \omega dx\\
&=&\displaystyle \langle GX,X\rangle_{\R^2}+2\dfrac{\omega^{p+2}(\ell)}{p+2}-2\dfrac{\omega^{p+2}(0)}{p+2}=\langle GX,X\rangle_{\R^2}\leq 0,
\end{array}
\end{equation}
where $G$ and $X$ were given in \eqref{dissipatividade_de_A}. Let us now define a Lyapunov function
$$\Phi(t)=\mathcal{E}(t)+\mu_1E_1(t)+\mu_2 E_2(t), \ t\geq 0,$$
where $E_1(t)$ and $E_2(t)$ are given by
\begin{equation*}
E_1(t)=\int_{\Omega} xu^2(x,t)dx \quad \text{and} \quad E_2(t)=|\nu_2|\int_{\Omega_0} \int_{\mathcal{M}}se^{-\delta \phi s}\sigma(s) z^2(t,\phi, s)dsd\phi,
\end{equation*}
$\mu_1$ and $\mu_2$ are positive constants to be determined and  $\delta>0$ is  arbitrary constant. Note that
$$ \mu_1 E_1(t)=\mu_1\int_{\Omega} xu^2(x,t)dx\leq \ell\mu_1\int_{\Omega} \omega^2(x,t)dx= \ell\mu_1\|\omega\|^2
$$
and
$$\mu_2 E_2(t)\leq   \mu_2|\nu_2|\int_{\Omega_0} \int_{\mathcal{M}}s\sigma(s) z^2(t,\phi, s)dsd\phi=\mu_2\| z(t)\|^2_{\ell^2(\mathcal{Q})}.$$
Consequently,
$$\mu_1E_1(t)+\mu_2E_2(t)\leq \max\{\ell\mu_1,\mu_2\}\mathcal{E}(t)$$
and, therefore
\begin{equation}
\label{eq3.8}
\mathcal{E}(t)\leq \Phi(t)\leq \left(1+\max\{\ell\mu_1,\mu_2\}\right)\mathcal{E}(t).
\end{equation}
Differentiating $E_1(t)$ and $E_2(t)$ using integration by parts and the boundary conditions of \eqref{sis1} and  \eqref{sis2.1_coupled}, we get
\begin{equation}
\label{eq3.9}
\begin{array}{rcl}
E_1'(t)&=&\displaystyle  \alpha\| \omega\|^2-3\beta \|\partial_x \omega\|^2-5\|\partial_x^2 \omega\|^2+\dfrac{2}{p+2}\int_{\Omega}  \omega^{p+2}dx\\
&&\displaystyle+\ell\left[\nu_1^2\left(\partial_x^2\omega(t,0)\right)^2+2\nu_1\nu_2\left(\partial_x^2\omega(t,0)\right)\left(\int_{\mathcal{M}}\sigma(s)z(t,1,s)ds\right)\right.\\
&&\displaystyle\left.+\nu_2^2\left(\int_{\mathcal{M}}\sigma(s)z(t,1,s)ds\right)^2\right]
\end{array}
\end{equation}
and
\begin{equation}
\label{eq3.10}
\begin{array}{rcl}
E_2'(t)&=&\displaystyle - |\nu_2|\int_{\mathcal{M}}e^{-\delta s}\sigma(s)\left( z(t,1,s)\right)^2 ds+|\nu_2|\left(\int_{\mathcal{M}}\sigma(s) ds\right)\left( \partial_x^2\omega(t,0)\right)^2\\
&&\displaystyle -|\nu_2|\int_{\mathcal{M}}\int_{\Omega_0}\delta s e^{-\delta\phi s}\sigma(s) z^2 d\phi ds.
\end{array}
\end{equation}
Thus, for $\Phi(t)=\mathcal{E}(t)+\mu_1E_1(t)+\mu_2E_2(t)$, we find that
\begin{equation*}
\begin{split}
 \Phi'(t)+&2\mu \Phi(t)
=  \langle GX,X\rangle_{\R^2}+\alpha\mu_1\| \omega\|^2-3\beta\mu_1 \|\partial_x \omega\|^2-5\mu_1\|\partial_x^2 \omega\|^2+\dfrac{2\mu_1}{p+2}\int_{\Omega}  \omega^{p+2}dx\\
&+\ell\mu_1\left[\nu_1^2\left(\partial_x^2\omega(t,0)\right)^2+2\nu_1\nu_2\left(\partial_x^2\omega(t,0)\right)\left(\int_{\mathcal{M}}\sigma(s)z(t,1,s)ds\right)\right.\\
&\left.+\nu_2^2\left(\int_{\mathcal{M}}\sigma(s)z(t,1,s)ds\right)^2\right]\\
& - \mu_2|\nu_2|\int_{\mathcal{M}}e^{-\delta s}\sigma(s)\left( z(t,1,s)\right)^2 ds+\mu_2|\nu_2|\left(\int_{\mathcal{M}}\sigma(s) ds\right)\left( \partial_x^2\omega(t,0)\right)^2\\
& -\mu_2|\nu_2|\int_{\mathcal{M}}\int_{\Omega_0}\delta s e^{-\delta\phi s}\sigma(s) z^2 d\phi ds +2\mu\|\omega(t)\|^2+2\mu \|z(t)\|^2_{L^2(\mathcal{Q })}+2\mu\mu_1\int_{\Omega} xu^2(x,t)dx\\
& +2\mu\mu_1 |\nu_2|\int_{\Omega_0} \int_{\mathcal{M}}se^{-\delta \phi s}\sigma(s) z(t,\phi, s)dsd\phi.
\end{split}
\end{equation*}
Next, let
$$G_{\mu_1}=\mu_1\ell\left(\begin{array}{cc}
\nu_1^2&\nu_1\nu_2\\
\nu_1\nu_2&\nu_2^2\\
\end{array}\right) , \ G_{\mu_2}=\mu_2\left(\begin{array}{cc}
|\nu_2| \displaystyle \int_{\mathcal{M}}\sigma(s)ds&0\\
0&0\\
\end{array}\right)
$$
and
$$X=\left(\begin{array}{c}
\partial_x^2 \omega(t,0)\\
\displaystyle \int_{\mathcal{M}}\sigma(s)z(t,1,s)ds\\
\end{array}\right).$$ Thus, we have that
\begin{equation*}
\begin{split}
 \Phi'(t)+&2\mu \Phi(t)
=  \langle (G+G_{\mu_1}+G_{\mu_2})X,X\rangle_{\R^2}+(\alpha\mu_1+2\mu)\| \omega\|^2-3\beta\mu_1 \|\partial_x \omega\|^2-5\mu_1\|\partial_x^2 \omega\|^2\\
&+\dfrac{2\mu_1}{p+2}\int_{\Omega}  \omega^{p+2}dx- \mu_2|\nu_2|\int_{\mathcal{M}}e^{-\delta s}\sigma(s)\left( z(t,1,s)\right)^2 ds\\
& -\mu_2|\nu_2|\int_{\mathcal{M}}\int_{\Omega_0}\delta s e^{-\delta\phi s}\sigma(s) z^2 d\phi ds +2\mu \|z(t)\|^2_{L^2(\mathcal{Q })}+2\mu\mu_1\int_{\Omega} xu^2(x,t)dx\\
& +2\mu\mu_1 |\nu_2|\int_{\Omega_0} \int_{\mathcal{M}}se^{-\delta \phi s}\sigma(s) z(t,\phi, s)dsd\phi\\
\leq&  \langle (G+G_{\mu_1}+G_{\mu_2})X,X\rangle_{\R^2}+\left(\alpha\mu_1+2\mu(1+\mu_1\ell)\right)\| \omega\|^2-3\beta\mu_1 \|\partial_x \omega\|^2-5\mu_1\|\partial_x^2 \omega\|^2\\
&+\dfrac{2\mu_1}{p+2}\int_{\Omega}  \omega^{p+2}dx- \mu_2|\nu_2|e^{-\delta \tau_2}\int_{\mathcal{M}}\sigma(s)\left( z(t,1,s)\right)^2 ds\\
& -\mu_2|\nu_2|e^{-\delta \tau_2}\delta\int_{\mathcal{M}}\int_{\Omega_0} s\sigma(s) z^2 d\phi ds\\
& +2\mu \|z(t)\|^2_{L^2(\mathcal{Q })}+2\mu\mu_1 |\nu_2|\int_{\Omega_0} \int_{\mathcal{M}}s\sigma(s) z(t,\phi, s)dsd\phi.
\end{split}
\end{equation*}
Now, observe that $$T(\mu_1,\mu_2):=G+G_{\mu_1}+G_{\mu_2}=G+\mu_1\ell\left(\begin{array}{cc}
\nu_1^2&\nu_1\nu_2\\
\nu_1\nu_2&\nu_2^2\\
\end{array}\right)+\mu_2\left(\begin{array}{cc}
|\nu_2| \int_{\mathcal{M}} \sigma(s)ds&0\\
0&0\\
\end{array}\right)$$
is a continuous map of $\R^2$ on the vector space of square matrices $M_{2\times 2}(\R)$ and that the determinant and trace are continuous functions of $M_{2\times 2} (\R)$ over $\R,$ we have that $h_1(\mu_1,\mu_2)=\det T(\mu_1,\mu_2)$ and $h_2(\mu_1,\mu_2)=\tr T(\mu_1,\mu_2)$ are continuous from $\R^2$ over $\R.$ Therefore, knowing that $h_1(0,0)=\det G>0$ and $h_2(0,0)=\tr G<0$ for $\mu_1,\mu_2$ small enough, one can claim that $h_1(\mu_1,\mu_2)>0$ and $h_2(\mu_1,\mu_2)<0.$ Thereby, $G+G_{\mu_1}+ G_{\mu_2}$ is negative defined for $\mu_1,\mu_2$ small enough.
Moreover, using the Poincar\'e inequality\footnote{$\|\omega\|^2\leq\dfrac{\ell^2}{\pi^2}\|\partial_x \omega\|^2$, for $\omega\in H_0^2(\Omega),$} we find
\begin{equation}
\label{eq3.16}
\begin{split}
 \Phi'(t)+2\mu \Phi(t)
\leq& \left[\dfrac{\ell^2}{\pi^2}\left(\alpha\mu_1+2\mu(1+\mu_1\ell)\right)-3\beta\mu_1\right]\|\partial_x \omega\|^2-5\mu_1\|\partial_x^2 \omega\|^2\\
&+\dfrac{2\mu_1}{p+2}\int_{\Omega}  \omega^{p+2}dx- \mu_2|\nu_2|e^{-\delta \tau_2}\int_{\mathcal{M}}\sigma(s)\left( z(t,1,s)\right)^2 ds\\
& +\left(2\mu(1+ \mu_1 |\nu_2|)-\mu_2|\nu_2|e^{-\delta \tau_2}\delta\right)\|z(t)\|^2_{L^2(\mathcal{Q })}.
\end{split}
\end{equation}

Now, we are going to estimate the integral $$ \dfrac{2\mu_1}{p+2}\int_{\Omega}  \omega^{p+2}dx.$$
For this, applying the Cauchy-Schwarz inequality and using the fact that the energy of the system $\mathcal{E}(t)$ is non-increasing, together with the embedding $H_0^1(\Omega)\hookrightarrow L^\infty(\Omega)$ we have, for $\|(\omega_0,z_0)\|_H<r$, that
\begin{equation}
\label{eq3.17}
\begin{array}{rcl}
\displaystyle \dfrac{2\mu_1}{p+2}\int_{\Omega}  \omega^{p+2}dx&\leq&\displaystyle \dfrac{2\mu_1}{p+2}\|\omega\|^2_{L^\infty(\Omega)}\int_{\Omega}  \omega^pdx\leq \dfrac{2\ell\mu_1}{p+2}\|\partial_x \omega\|^2\int_{\Omega}  \omega^pdx\\
&\leq&\displaystyle  \dfrac{2\ell\mu_1}{p+2}\|\partial_x \omega\|^2 \ell^{1-\frac{p}{2}}\|\omega\|^p\leq \dfrac{2\ell^{2-\frac{p}{2}}\mu_1}{p+2}\|\partial_x \omega\|^2\|(\omega_0,z_0)\|_H^p\\
&\leq&\displaystyle  \dfrac{2\ell^{2-\frac{p}{2}}\mu_1r^p}{p+2}\|\partial_x \omega\|^2.
\end{array}
\end{equation}
Combining \eqref{eq3.17} and \eqref{eq3.16} yields
\begin{equation}
\label{eq3.18}
\begin{split}
  \Phi'(t)+2\mu \Phi(t)\leq & \left[\dfrac{\ell^2}{\pi^2}\left(\alpha\mu_1+2\mu(1+\mu_1\ell)\right)-3\beta\mu_1\right]\|\partial_x \omega\|^2-5\mu_1\|\partial_x^2 \omega\|^2\\
&+\dfrac{2\ell^{2-\frac{p}{2}}\mu_1r^p}{p+2}\|\partial_x \omega\|^2- \mu_2|\nu_2|e^{-\delta \tau_2}\int_{\mathcal{M}}\sigma(s)\left( z(t,1,s)\right)^2 ds\\
& +\left(2\mu(1+ \mu_1 |\nu_2|)-\mu_2|\nu_2|e^{-\delta \tau_2}\delta\right)\|z(t)\|^2_{L^2(\mathcal{Q })}\\
\leq&  \left[\dfrac{\ell^2}{\pi^2}\left(\alpha\mu_1+2\mu(1+\mu_1\ell)\right)-3\beta\mu_1+\dfrac{2\ell^{2-\frac{p}{2}}\mu_1r^p}{p+2}\right]\|\partial_x \omega\|^2-5\mu_1\|\partial_x^2 \omega\|^2\\
& +\left(2\mu(1+ \mu_1 |\nu_2|)-\mu_2|\nu_2|e^{-\delta \tau_2}\delta\right)\|z(t)\|^2_{L^2(\mathcal{Q })}.
\end{split}
\end{equation}
Note that $\Phi'(t)+2\mu \Phi(t)<0$ when $$2\mu(1+ \mu_1 |\nu_2|)-\mu_2|\nu_2|e^{-\delta \tau_2}\delta<0$$ and $$\dfrac{\ell^2}{\pi^2}\left(\alpha\mu_1+2\mu(1+\mu_1\ell)\right)-3\beta\mu_1
+\dfrac{2\ell^{2-\frac{p}{2}}\mu_1r^p}{p+2}<0,$$
which holds for $\mu>0$ satisfying, respectively
$$\mu<\dfrac{\mu_2 |\nu_2|e^{-\delta\tau_2}\delta}{2(1+\mu_1|\nu_2|)}$$
and
$$0<\mu<\dfrac{\mu_1}{2\ell^2(1+\ell\mu_1)(p+2)}\left[(p+2)(3\pi^2\beta-\alpha \ell^2)-2\pi^2\ell^{2-\frac{p}{2}}r^p\right],$$
where we need to take $r>0$ satisfying
$$(p+2)(3\pi^2\beta-\alpha \ell^2)-2\pi^2\ell^{2-\frac{p}{2}}r^p>0$$
or, equivalently, $r>0$ must satisfy
$$r<\left(\dfrac{(p+2)(3\pi^2\beta-\alpha \ell^2)}{2\pi^2\ell^{2-\frac{p}{2}}}\right)^\frac{1}{p}.$$
Thus, for $\mu_1,\mu_2$ small enough and an arbitrary $\delta>0$, taking $$r<\left(\dfrac{(p+2)(3\pi^2\beta-\alpha \ell^2)}{2\pi^2\ell^{2-\frac{p}{2}}}\right)^\frac{1}{p}$$ and $$\mu<\min\left\{\dfrac{\mu_2 |\nu_2|e^{-\delta\tau_2}\delta}{2(1+\mu_1|\nu_2|)},
\dfrac{\mu_1}{2\ell^2(1+\ell\mu_1)(p+2)}\left[(p+2)(3\pi^2\beta-\alpha \ell^2)-2\pi^2\ell^{2
-\frac{p}{2}}r^p\right]\right\},$$
we get that
$$\Phi'(t)+2\mu \Phi(t)<0 \iff \Phi(t)\leq \Phi(0)e^{-2\mu t}.$$
Lastly, from \eqref{eq3.8}, we get
$$\mathcal{E}(t)\leq \Phi(t)\leq \Phi(0)e^{-2\mu t}\leq (1+\max\{\ell\mu_1,\mu_2\})E(0)e^{-2\mu t}\leq \kappa E(0)e^{-2\mu t},$$
for $\kappa>1+\max\{\ell\mu_1,\mu_2\}$, proving the theorem. \qed

\subsection{Proof of Theorem \ref{comp}}
First, we deal with the linear system \eqref{sis2.1_coupled} and claim that the following observability inequality holds
\begin{equation}\label{OI}
\|\omega_{0}\|^2+\| z_{0}\|_{L^2 (\mathcal{Q})}^2 \leq
C \int_{0}^{T}\left((\partial^2_x \omega(t,0))^{2}+\int_{\mathcal{M}} s \sigma(s)   z^2 (t,1,s) \, ds \right) d t,
\end{equation}
where  $\left(\omega_{0}, z_{0}\right) \in H$ and $(\omega, z)(t)=e^{t A} \left(\omega_{0}, z_{0}\right)$ is the unique solution of \eqref{sis2.1_coupled}. This leads to the exponential stability in $H$ of the solution $(y,z)$ to \eqref{sis2.1_coupled}. The proof of this inequality can be obtained by a contradiction argument. Indeed, if \eqref{OI} is not true, then there exists a sequence  $\{\left(\omega_{0}^{n}, z_{0}^{n}\right)\}_{n} \subset H$ such that
\begin{equation}\label{unit}
\|\omega_{0}^{n}\|^{2}+\|z_{0}^{n}\|^{2}_{L^2 (\mathcal{Q})}=1
\end{equation}
and
\begin{equation}\label{n=0}
\left\| \partial^2_x \omega^{n}( \cdot, 0)\right\|_{L^{2}(0, T)}^{2}+\int_{\mathcal{M}} s \sigma(s)   z^2 (t,1,s) \, ds \rightarrow 0 \text { as } n \rightarrow+\infty,
\end{equation}
where $\left(\omega^{n}, z^{n}\right)(t)=e^{tA} \left(\omega_{0}^{n}, z_{0}^{n}\right)$.
Then, arguing as in \cite{luan}, we can deduce from Proposition \ref{linear} that  $\{\omega^{n}\}_{n}$ is convergent in $L^{2}\left(0, T, L^{2}(\Omega)\right)$. Moreover, $\{\omega_{0}^{n}\}_{n}$ is a Cauchy sequence in $L^{2}(\Omega)$, while $\{z_{0}^{n}\}_n$ is a Cauchy sequence in $L^{2}(\mathcal{Q})$. Thereafter, let $\left(\omega_{0}, z_{0}\right)=\lim _{n \rightarrow \infty}\left(\omega_{0}^{n}, z_{0}^{n}\right) \;\; \mbox{in} \; H$ and hence
$\|\omega_{0}\|^{2}+\|z_{0}\|^{2}_{L^2 (\mathcal{Q})}=1,$ by virtue of \eqref{unit}. Next, take $(\omega,z)=e^{\cdot A} \left(\omega_{0}, z_{0}\right),$ and assume, for the sake of simplicity and without loss of generality, that $\alpha=\beta=1$. This, together with Proposition \ref{linear} and \eqref{n=0}, implies that $\omega$ is solution of the system
$$
\begin{cases}
\partial_t \omega + \partial_x \omega+\partial^3_x \omega-\partial^5_x \omega=0, & x \in \Omega, t>0, \\
\omega(0, t)=\omega(\ell, t)=\partial_x \omega(\ell, t)=\partial_x \omega(0, t)=\partial^2_x \omega(\ell, t)=\partial^2_x \omega(0, t)=0, & t>0, \\
\omega(x, 0)=\omega_{0}(x), & x \in \Omega,
\end{cases}
$$
with
$\left\|\omega_{0}\right\|_{L^{2}(\Omega)}=1 .$
The latter contradicts the result obtained in \cite[Lemma 4.2]{luan}, which states that the above system has only the trivial solution (see also Lemma \ref{lem2}). This proves the observability inequality \eqref{OI}.

Now, let us go back to the original system \eqref{sis1} and use the same arguments as in \cite{ro}. First, we restrict ourselves to the case $p=1$ as the case $p=2$ is similar. Next, consider an initial condition $\left\|\left(\omega_{0}, z_{0}\right)\right\|_H \leq \varrho,$
where $\varrho$ will be fixed later. Then, the solution $\omega$ of \eqref{sis1}  can be written as $\omega=\omega_1+\omega_2$, where  $\omega_1$ is the solution of \eqref{sis2.1_coupled} with the initial data $\left(\omega_{0}, z_{0}\right)\in H$ and $\omega_2$ is solution of \eqref{sis_coupled_nh} with null data and right-hand side $\varphi=\omega \partial_x \omega \in L^1(0,T;L^2(\Omega))$,  as in Lemma \ref{teor2}. In other words, $\omega_1$ is the solution of
$$\begin{cases}\partial_t \omega_{1}-\partial^5_x \omega_1+ \partial^3_x \omega_{1}+ \partial_x \omega_{1}=0, & x \in \Omega, t>0, \\
\omega_1(t,0)=\omega_1(t,\ell)=\partial_x \omega_1(t,0)=\partial_x \omega_1(t,\ell)=0, & t>0, \\
\partial^2_x \omega_{1}(t,\ell)=\nu_1 \partial^2_x \omega_{1}(t,0)+\nu_2 \displaystyle \int_{t-\tau_2}^{t-\tau_1} \sigma(t-s)  \partial_x^2 \omega (s,0) \, ds, & t>0, \\
\partial^2_x \omega_{1}(t,0)=z_{0}(t), & t \in(-\tau_2, 0), \\
\omega_1(0,x)=\omega_{0}(x), & x \in \Omega,\end{cases}$$
and $\omega_{2}$ is solution of
$$\begin{cases}
\partial_t \omega_{2}-\partial^5_x \omega_2+\partial^3_x \omega_{2}+ \partial_x \omega_{2}=-\omega \partial_x \omega, & x \in \Omega, t>0, \\
\omega_{2}(t,0)=\omega_{2}(t,\ell)=\partial_x \omega_2(t,0)=\partial_x \omega_2(t,\ell)=0, & t>0, \\
\partial^2_x \omega_{2}(t,\ell)=\nu_1 \partial^2_x \omega_{2}(t,0)+\nu_2 \displaystyle \int_{t-\tau_2}^{t-\tau_1} \sigma(t-s)  \partial_x^2 \omega (s,0) \, ds, & t \in(-\tau_2, 0), \\
\partial^2_x \omega_{2}(t,0)=0, & x \in \Omega,  \\
\omega_{2}(0,x)=0, & x \in\Omega.
\end{cases}$$
In light of the exponential stability of the linear system \eqref{sis2.1_coupled} (see the beginning of this subsection) and Theorem \ref{teor2}, we have
\begin{equation}\label{31}
\|(\omega(T), z(T))\|_{H}
\leq  \chi \left\|\left(\omega_{0}, z_{0}\right)\right\|_{H}+C\|\omega\|_{L^{2}\left(0, T, H^{2}(\Omega)\right)}^{2},
\end{equation}
in which $\chi \in(0,1)$. Subsequently, multiply \eqref{sis1}$_1$ by $xu$ and performing the same computations as for \eqref{eq3.9}, we get
\begin{equation}\label{q1}
\begin{split}
&\int_{\Omega} x \omega^2(T,x) d x+3 \int_{0}^{T}
\int_{\Omega}\left( \partial_x \omega(t,x)\right)^{2} d x d t+5\int_{0}^{T} \int_{\Omega}\left( \partial^2_xu(t,x)\right)^{2} d x dt = \\&
 \int_{0}^{T} \int_{\Omega} \omega^2(t,x) d x dt+\ell \int_0^T \left(\nu_1\partial_x^{2}\omega(t,0) +\nu_2  \int_{\mathcal{M}} \sigma(s) z(t,1,s) \, ds\right)^2 dt+ \int_{\Omega} x  \omega_{0}^2(x)  d x \\
 &+\frac{2}{3} \int_{0}^{T} \int_{\Omega} \omega^3 (t,x) dxdt.
\end{split}
\end{equation}
On one hand, multiplying the first equation of \eqref{sis1} by $\omega$ and arguing as done for \eqref{est1} (see \eqref{eq2.23}), we get
\begin{equation}\label{q2}
\displaystyle  \int_0^T \left(\nu_1\partial_x^{2}\omega(t,0) +\nu_2  \int_{\mathcal{M}} \sigma(s) z(t,1,s) \, ds\right)^2 dt \leq  C  \| (\omega_0, z_0 )\|_H^2.
\end{equation}
On the other hand, using Gagliardo–Nirenberg and Cauchy-Schwarz inequalities, together with the dissipativity of the system \eqref{sis1}, we deduce that
\begin{equation*}	
		\int_{0}^{T} \int_{\Omega} \omega^{3} d x d t  \leq C(T)\left\|(\omega_{0},z_0)\right\|^{2}_{H} \|\omega\|_{L^{2}\left(0, T ; H^{2}(\Omega)\right)}.
\end{equation*}
Applying Young's inequality to the last estimate and combining the obtained result with \eqref{q1}-\eqref{q2}, we reach
\begin{equation}\label{33}
\|\omega\|_{L^{2}\left(0, T ; H^{2}(\Omega)\right)}^{2}\leq C\left\|\left(\omega_{0}, z_{0}\right)\right\|_{H}^{2} \left(1+\left\|\left(\omega_{0}, z_{0}\right)\right\|_{H}^{2}\right).
\end{equation}
Finally, recalling that $\left\|\left(\omega_{0}, z_{0}\right)\right\|_H \leq \varrho,$ and inserting \eqref{33} into \eqref{31}, we get
$$
\|(\omega(T), z(T))\|_{H} \leq\left\|\left(\omega_{0}, z_{0}\right)\right\|_{H}\left(\chi+C \varrho+C \varrho^{3}\right).
$$
Given $\eta>0$ sufficiently small so that $\chi+\eta<1$, one can choose $\varrho$ small such that $\varrho+\varrho^{3}<\frac{\eta}{C}$, to obtain
$$
\|(\omega(T), z(T))\|_{H} \leq(\chi+\eta)\left\|\left(\omega_{0}, z_{0}\right)\right\|_{H}.
$$
Lastly, using the semigroup property and the fact that $\chi+\eta<1$, we conclude the exponential stability result of Theorem \ref{comp}.  \qed

\section{Conclusion} This article presented a study on the stability of the Kawahara equation with a boundary-damping control of finite memory type. It is shown that such a control is good enough to obtain the desirable property, namely, the exponential decay of the system's energy. The proof is based on two different approaches. The first one invokes a Lyapunov functional and provides an estimate of the energy decay. In turn, the second one uses a compactness-uniqueness argument that reduces the issue to a spectral problem. 

Finally, we would like to point out that our well-posedness result (see Theorem \ref{theorem3}) is shown for the nonlinearity ${\omega^p} \partial_x \omega$, where $p \in \{1,2\}$. Notwithstanding, we believe that using an interpolation argument, this finding should remain valid if $p \in (1,2)$. The same remark applies to the second stability result (see Theorem \ref{comp}). It is also noteworthy that our first stability outcome (see Theorem \ref{Lyapunov}) is established for a more general nonlinearity ${\omega^p} \partial_x \omega$, $p \in [1,2]$.

\subsection*{Acknowledgments}
This work is part of the Ph.D. thesis of de Jesus at the Department of Mathematics of the Federal University of Pernambuco and was done while the first author was visiting Virginia Tech. The first author thanks the host institution for their warm hospitality.

\subsection*{Authors contributions:} Capistrano-Filho, Chentouf and de Jesus work equality in Conceptualization; formal analysis; investigation; writing--original draft; writing--review and editing.

\subsection*{Conflict of interest statement} This work does not have any conflicts of interest.

\subsection*{Data availability statement} Is not applicable to this article as no new data were created or analyzed in this study.

\end{document}